\documentclass[10pt]{article}
\usepackage{color}
\usepackage{amssymb}
\usepackage{amsthm,array,amssymb,amscd,amsfonts,latexsym, url}
\usepackage{amsmath}
\usepackage[all]{xy}
\newtheorem{theo}{Theorem}[section]
\newtheorem{prop}[theo]{Proposition}
\newtheorem{claim}[theo]{Claim}
\newtheorem{lemm}[theo]{Lemma}
\newtheorem{coro}[theo]{Corollary}

\title{Footnotes to papers of O'Grady and Markman}
\author{Claire Voisin\footnote{The author is supported by the ERC Synergy Grant HyperK (Grant agreement No. 854361).}}

\date{}

\newfont{\gothic}{eufb10}
\begin{document}
\maketitle
\setcounter{section}{-1}

\begin{flushright} {\it Pour Olivier,  avec amiti\'{e}}
\end{flushright}
\begin{abstract} In this paper, we first  generalize to any hyper-K\"{a}hler manifold $X$ with $b_3(X)\not=0$ results proved by O'Grady for hyper-K\"{a}hler manifolds of generalized Kummer type. In the second part, we restrict to hyper-K\"{a}hler manifolds of generalized Kummer type and prove, using results of Markman, that their  Kuga-Satake correspondence is algebraic.
 \end{abstract}
\section{Introduction}
This paper provides complements to the recent papers \cite{ogrady} by O'Grady  and  \cite{markman} by Markman. Hyper-K\"{a}hler manifolds $X$ of generalized Kummer type are obtained by deforming the generalized Kummer varieties $K_n(A)$ constructed by Beauville \cite{beauville} starting from an abelian surface $A$. The manifold $K_n(A)$ is defined as the subset of  the punctual Hilbert scheme $A^{[n+1}]$ consisting of $0$-dimensional subschemes with trivial Albanese class. For $n\geq 2$, one  has $b_2(X)=7,\,b_3(X)=8$. Both papers are concerned with the intermediate Jacobians $J^3(X)$ for  $X$ as above. Recall that $J^3(X)$ is the complex torus built from the Hodge structure on $H^3(X,\mathbb{Z})$, which in this case is of level $1$ since $H^{3,0}(X)=0$, and is thus an abelian variety when $X$ is projective. As
$b_3(X)=8$,  $J^3(X)$ is an abelian fourfold. O'Grady proves the following results.
\begin{theo} (O'Grady \cite{ogrady})  \label{theoogradyintro} (1)  $J^3(X)$ is a Weil abelian fourfold.

(2) For a very general projective deformation of $X$,  the Kuga-Satake abelian variety $KS(X)$ of $(H^2(X,\mathbb{Q})_{\rm tr},(\,,\,))$ is isogenous to a power of   $J^3(X)$.
\end{theo}

Let us explain both statements. A Weil abelian fourfold is an abelian fourfold that admits an endomorphism $\phi: A\rightarrow A$ satisfying a quadratic equation $\phi^2=-d {\rm Id}$, with $d>0$, with the following extra condition: consider the action $\phi_\mathbb{C}$ of $\phi$ on $H^1(X,\mathbb{C})$ by pullback. Then $\phi_\mathbb{C}$ preserves $H^{1,0}(A)$ and $H^{0,1}(A)$  since it is  a morphism of Hodge structures and thus it has eigenvalues either  $i\sqrt{d}$ or $-i \sqrt{d}$ on these $4$-dimensional spaces. The Weil condition is that
$\phi_\mathbb{C}$ has both  eigenvalues   $i\sqrt{d}$ and  $-i \sqrt{d}$ with multiplicity $2$ on  $H^{1,0}(A)$ (hence also on $H^{0,1}(A)$). It guarantees that $A$ has a $2$-dimensional space of Weil Hodge classes of degree $4$. More precisely, denoting $K$ the number field $\mathbb{Q}[\sqrt{-d}]$, $H^1(A,\mathbb{Q})$ is a $4$-dimensional $K$-vector space and  the condition above guarantees that the
$2$-dimensional subspace
$$\bigwedge^4_KH^1(A,\mathbb{Q}) \subset  \bigwedge^4H^1(A,\mathbb{Q}) $$
consists of classes of Hodge type $(2,2)$, hence of Hodge classes.

Concerning the point (2), let us define  a Hodge structure of hyper-K\"{a}hler type as the data of weight $2$ (effective, rational or integral) Hodge structure
$(H^2,F^iH^2_\mathbb{C})$ with $h^{2,0}=1$, equipped with a nondegenerate quadratic form satisfying the first Hodge-Riemann relations, namely
$$ (H^{2,0},F^1 H^2_\mathbb{C})=0$$
and the condition that  the restriction of $(\,,\,)$ to the real $2$-plane
$(H^{2,0}\oplus H^{0,2})\cap H^2_\mathbb{R}$ is positive definite. This is the structure one gets on the degree $2$ cohomology of a hyper-K\"{a}hler manifold, the intersection pairing being given by the Beauville-Bogomolov quadratic form. We will say that we have a polarized Hodge structure of hyper-K\"{a}hler type if furthermore $(\,,\,)$ is negative definite on the space $H^{1,1}_\mathbb{R}:=H^{1,1}\cap H^2_\mathbb{R}$. We get such a structure by considering the transcendental   degree $2$ cohomology of a projective hyper-K\"{a}hler manifold. The Kuga-Satake variety $KS(H^2,(\,,\,))$ associated to a Hodge structure $(H^2,(\,,\,))$ of hyper-K\"{a}hler type is a complex torus, or weight $1$ Hodge structure (which is defined up to isogeny if we work with rational Hodge structures)  built by a formal process (see \cite{Kugasatake}, \cite{deligne} and Section \ref{secKS}). If the considered Hodge structure is a polarized Hodge structure of hyper-K\"{a}hler type, $KS(H^2,(\,,\,))$ is  an abelian variety.
In the case of  a very general  polarized  weight $2$ Hodge structure of generalized Kummer type, we have  $b_{2,tr}=6$ and  the corresponding Kuga-Satake variety is isogenous to a power of a Weil abelian fourfold.   The point (1) thus  follows from (2)  and (2) itself is a consequence of  a certain  universality property of the Kuga-Satake construction proved in \cite{charles}, (and \cite{vgeemen-voisin} in a slightly different  setting, see  Section \ref{secKS}),  and of  the following result.
\begin{theo} \label{thog}(O'Grady \cite{ogrady}) Let $X$ be a hyper-K\"{a}hler $2n$-fold of  generalized Kummer deformation type with $n\geq2$. Then the composite map
map
$$\bigwedge^2H^3(X,\mathbb{Q})\rightarrow H^6(X,\mathbb{Q})\stackrel{Q_X^{n-2}}{\rightarrow} H^{4n-2}(X,\mathbb{Q})$$
is surjective.
\end{theo}

Here $Q_X\in H^4(X,\mathbb{Q})$ is a cohomology class which is constructed using the Beauville-Bogomolov form (see Section \ref{seclef}).
Our first result is the following generalization of  Theorem \ref{thog}.
\begin{theo} \label{theointro}  Let $X$ be a hyper-K\"{a}hler $2n$-fold such that $b_3(X)\not=0$.  Then

(1)  The composite
map
$$\bigwedge^2H^3(X,\mathbb{Q})\rightarrow H^6(X,\mathbb{Q})\stackrel{Q_X^{n-2}}{\rightarrow} H^{4n-2}(X,\mathbb{Q})$$
is surjective.

(2) Assuming $X$ is projective, the intermediate Jacobian $J^3(X)$  contains a simple component of the Kuga-Satake abelian variety of
$H^2(X,\mathbb{Q})_{\rm tr}$.

(3) One has $b_3(X)\geq 2^k$, where $k= \frac{b_2(X)-2}{2}$ if $b_2(X)$ is even,  $k= \frac{b_2(X)-1}{2}$ if $b_2(X)$ is odd.
\end{theo}
We will also prove  similar results, and in particular the bound
\begin{eqnarray}\label{eqboundb2n-1} b_{2n-1}(X)\geq 2^k, \,\,{\rm if}\,\,H^{2n-1}(X,\mathbb{Q})\not=0,\,\,H^{2n-3}(X,\mathbb{Q})=0.
\end{eqnarray}
where $k$ is as in (3).  In particular, we get the following corollary in dimension $6$.
\begin{coro} Let $X$ be a hyper-K\"{a}hler $6$-fold such that $b_{\rm odd}(X)\not=0$. Then $b_{\rm odd}(X)\geq 2^k$, where $k= \frac{b_2(X)-2}{2}$ if $b_2(X)$ is even,  $k= \frac{b_2(X)-1}{2}$ if $b_2(X)$ is odd.
\end{coro}

The Betti numbers of hyper-K\"{a}hler  manifolds have been studied  in  \cite{salamon}, \cite{guan} which establishes very precise bounds in dimension $4$ and  in  \cite{sawon} which claims similar bounds in dimension 6 (but the proof seems to be incomplete). The paper \cite{kimlaza}  gives very precise conjectural  bounds (for example bounds  on $b_2$ depending only on the dimension), depending on a conjecture on  the Looijenga-Lunts representation \cite{loo}, \cite{verb}.  The subject  remains however wide open.

  A key point in both cases  is the fact that the weight $3$ or weight $2n-1$ Hodge structure one considers is of Hodge level $1$, that is,  they satisfy the property $h^{p,q}=0$ for $|p-q|>1$. As we already mentioned, the points (2) and (3) follow, using this observation,  from the point  (1), and from a universality property for the Kuga-Satake weight $1$ Hodge structure, proved by Charles in \cite{charles}, even in the unpolarized case.

The second part of this paper provides a complement to Markman's paper \cite{markman}. In this paper, Markman proves the Hodge conjecture for the Weil Hodge classes on the Weil abelian fourfolds appearing in Theorem \ref{theoogrady}. He also proves that the Abel-Jacobi map
$$\Phi_X:{\rm CH}^2(X)_{\rm alg}\rightarrow J^3(X),$$
defined on the group of codipmension 2 cycles on $X$ algebraically equivalent  to $0$,
is surjective for a projective hyper-K\"{a}hler manifold $X$ of generalized Kummer deformation type. This statement was expected as a consequence of  the generalized Hodge conjecture because $H^{3,0}(X)=0$ (see \cite{voisintorino}).

Our second result is the following
\begin{theo}  For $X$ a projective hyper-K\"{a}hler manifold of generalized Kummer deformation type with $n\geq2$, the Kuga-Satake correspondence between $X$ and its Kuga-Satake variety $KS(X)$  is algebraic.
\end{theo}
To explain this statement, the Kuga-Satake construction  in the polarized case produces  an abelian variety $KS(X)$ associated to the polarized Hodge structure
$(H^2(X,\mathbb{Q})_{\rm tr},(\,,\,))$  which has the property that $H^2(X,\mathbb{Q})_{\rm tr}$ is a  Hodge substructure of $H^2(KS(X),\mathbb{Q})$.  The Hodge conjecture predicts the existence of a correspondence between $X$ and $KS(X)$, that is an algebraic cycle $\Gamma$ of codimension $2n$ with $\mathbb{Q}$-coefficients   in $X\times KS(X)$,  such that $\Gamma_*$ induces the given embedding
$H^2(X,\mathbb{Q})_{\rm tr} \hookrightarrow H^2(KS(X),\mathbb{Q})$.  The meaning of the ``algebraicity of the Kuga-Satake correspondence"  is the  existence of such cycle $\Gamma$ (see \cite{geemen} for a general discussion).

The algebraicity of the Kuga-Satake correspondence is known for projective $K3$ surfaces with Picard number at least $17$ \cite{morrison}. It is also known by work of Paranjape \cite{paranjape} for $K3$ surfaces with Picard number $16$ obtained as desingularizations of  double covers of $\mathbb{P}^2$ ramified along $6 $ lines. Some hyper-K\"{a}hler examples involving cubic fourfolds  have been exhibited in \cite{voisinsym}.
\section{Applications of the hard Lefschetz theorem}

\subsection{Degree 3 cohomology: complement to a paper of O'Grady \label{seclef}}
 Let $X$ be a hyper-K\"{a}hler manifold of dimension $2n$ with $n\geq 2$. The Beauville-Bogomolov quadratic form $q_X$ is a nondegenerate quadratic form on $H^2(X,\mathbb{Q})$, whose inverse gives an element of ${\rm Sym}^2H^2(X,\mathbb{Q})$. By Verbitsky \cite{verbitsky}, the later space imbeds by cup-product in $H^4(X,\mathbb{Q})$, hence we get a class
 \begin{eqnarray} \label{eqQX}Q_X\in H^4(X,\mathbb{Q}).
 \end{eqnarray}
  The O'Grady map
 $ \phi:\bigwedge^2H^3(X,\mathbb{Q})\rightarrow H^{4n-2}(X,\mathbb{Q})$
 is defined by \begin{eqnarray}\label{eqogradymap}\phi(\alpha\wedge \beta)=Q_X^{n-2}\cup \alpha\cup \beta. \end{eqnarray}

 The following result was first  proved by O'Grady \cite{ogrady} in the case of a hyper-K\"{a}hler manifold of generalized Kummer deformation type.
\begin{theo} \label{theosurj3} Let $X$ be a hyper-K\"{a}hler manifold of dimension $2n$. Assume $H^3(X,\mathbb{Q})\not=0$. Then
the O'Grady map map $\phi: \bigwedge^2H^3(X,\mathbb{Q})\rightarrow H^{4n-2}(X,\mathbb{Q})$ is surjective.
\end{theo}
\begin{proof}  We can choose the complex structure on $X$ to be general, so that $\rho(X)=0$, and this implies that the Hodge structure on $H^2(X,\mathbb{Q})$ (or equivalently $H^{4n-2}(X,\mathbb{Q})$ as they are isomorphic  by combining Poincar\'{e} duality and the self-duality given by the  Beauville-Bogomolov form) is simple. As the morphism $\phi$ is a morphism of Hodge structures, its image is a Hodge substructure of $H^{4n-2}(X,\mathbb{Q})$, hence either $\phi$ is surjective, or it is $0$. Theorem \ref{theosurj3} thus follows fropm the next proposition.\end{proof}
\begin{prop}\label{prononzero}  The map $\phi$ is not identically $0$.
\end{prop}
\begin{proof} Let $\omega\in H^2(X,\mathbb{R})$ be a K\"{a}hler class. Then we know that the $\omega$-Lefschetz intersection pairing
$\langle \,,\,\rangle_\omega$ on $H^3(X,\mathbb{R})$, defined by
$$\langle \alpha,\beta\rangle_\omega:=\int_X\omega^{2n-3}\cup\alpha\cup \beta$$
is nondegenerate. This implies that
the cup-product map
$$\psi:\bigwedge^2H^3(X,\mathbb{Q})\rightarrow H^6(X,\mathbb{Q})$$
has the property that ${\rm Im}\,\psi$ pairs nontrivially with the image of the map
$${\rm Sym}^{2n-3}H^2(X,\mathbb{Q})\rightarrow H^{4n-6}(X,\mathbb{Q})$$
given by cup-product.
Note that the Hodge structure on $H^3(X,\mathbb{Q})$ has Hodge level $1$, so that the Hodge structure on the  image of  ${\rm Im}\,\psi$ in
${\rm Sym}^{2n-3}H^2(X,\mathbb{Q})^*$ is a Hodge structure of level at most $2$.
We now argue as in  \cite{voisinisot}. We choose $X$ very general so that the Mumford-Tate group of the Hodge structure on $H^2(X,\mathbb{Q})$ is the orthogonal group $O(q_X)$. Any
Hodge substructure of ${\rm Sym}^{2n-3}H^2(X,\mathbb{Q})^*\cong {\rm Sym}^{2n-3}H^2(X,\mathbb{Q})$ is thus a direct sum of $O(q_X)$-subrepresentations of ${\rm Sym}^{2n-3}H^2(X,\mathbb{Q})$.
Elementary representation theory of $O(q_X)$ then shows  that  the irreducible $O(q_X)$-subrepresentations of
${\rm Sym}^{2n-3}H^2(X,\mathbb{Q})$ are the subspaces
$$Q_X^l {\rm Sym}^{2n-3-2l}H^2(X,\mathbb{Q})^0,$$
where we see here $Q_X$ as an element of ${\rm Sym}^2H^2(X,\mathbb{Q})$, and
$${\rm Sym}^{k}H^2(X,\mathbb{Q})^0\subset {\rm Sym}^{k}H^2(X,\mathbb{Q})$$
can be  defined after passing to $\mathbb{C}$-coefficients as the subspace  generated by
$ \alpha^{k}$ with  $q_X(\alpha)=0$ (this definition  is correct with $\mathbb{Q}$-coefficients only if   the quadratic form $q_X$ has a zero).

The irreducible Hodge structure on $Q_X^l {\rm Sym}^{2n-3-2l}H^2(X,\mathbb{Q})^0$ has Hodge level $>2$ when $2n-3-2l>1$ since it contains the class
$Q_X^l\sigma_X^{2n-3-2l}$ which is of type $ (4n-6-4l,2l)$, where $\sigma_X$ generates $H^{2,0}(X)$. It follows that  ${\rm Im}\,\psi$ can pair nontrivially only with
$Q_X^{n-2} H^2(X,\mathbb{Q})$, hence the map $Q_X^{n-2}\psi$, which is the O'Grady map, is nonzero,  which concludes the proof.
\end{proof}

\subsection{Cohomology  of degree $2n-1$}
For other odd degree $2k-1\leq 2n-1$, one may wonder what  the hard  Lefschetz theorem gives.
The proof of Proposition \ref{prononzero} will give as well:
\begin{prop} \label{prosur2n-1} The composition
$$\psi': \bigwedge^2 H^{2k-1}(X,\mathbb{Q})\rightarrow H^{4k-2}(X,\mathbb{Q})\rightarrow {\rm Sym}^{2n-2k+1} H^2(X,\mathbb{Q})^*,$$
where the first map is the  cup-product and the second one is Poincar\'{e} dual to the cup-product map
${\rm Sym}^{2n-2k+1} H^2(X,\mathbb{Q})\rightarrow H^{4n-4k+2}(X,\mathbb{Q})$, is nontrivial (and even, nondegenerate).
\end{prop}
However, we do not know a priori the Hodge level of $H^{2k-1}(X,\mathbb{Q})$ so we do not know to  which irreducible component  of the $O(q)$-representation of ${\rm Sym}^{2n-2k+1} H^2(X,\mathbb{Q})^*$ the image  ${\rm Im}\,\psi'$ can map nontrivially.
In  the case of degree $2k-1=2n-1$, we have  only one piece, namely $ H^2(X,\mathbb{Q})^*$, hence we get:
\begin{coro} \label{corsur2n-1} If $X$ is a hyper-K\"{a}hler manifold of dimension $2n$ with $H^{2n-1}(X,\mathbb{Q})\not=0$, the cup-product map
$$\bigwedge^2 H^{2n-1}(X,\mathbb{Q})\rightarrow  H^2(X,\mathbb{Q})^*.$$
is surjective.
\end{coro}
\begin{proof} Indeed, the map $\psi$ is nonzero and a morphism of Hodge structures, the right hand side being a simple Hodge structure for a very general complex structure on $X$.
\end{proof}
We will also use in next section the following observation.
\begin{lemm} \label{ledecomp} Let $X$ be a hyper-K\"{a}hler $2n$-fold. Then
the Hodge structure on the quotient

\begin{eqnarray}\label{eqdecomp}H^{2n-1}(X,\mathbb{Q})^0:=H^{2n-1}(X,\mathbb{Q})/ H^2\cup H^{2n-3}(X,\mathbb{Q})
\end{eqnarray}
 has Hodge level $1$. In particular, if $H^{2n-3}(X,\mathbb{Q})=0$, the Hodge structure on $H^{2n-1}(X,\mathbb{Q})$ has Hodge level $1$.
\end{lemm}
\begin{proof} The statement is that the $(p,q)$-components of $H^{2n-1}(X,\mathbb{C})^0$ vanish unless $(p,q)=(n,n-1)$ or $(p,q)=(n-1,n)$.
We thus have to show that any class in $H^{p,q}(X)$ with $p>n$ or $q>n$ belongs to $H^2(X,\mathbb{C})\cup H^{2n-3}(X,\mathbb{C})$.
 This follows  from
the fact that the cup-product map by $\sigma_X^l$ induces a vector bundle isomorphism
$$\sigma_X^l\wedge: \Omega_X^{n-l}\rightarrow \Omega_X^{n+l},$$
hence an isomorphism $\sigma_X^l\cup: H^{n-l,q}(X)\cong H^{n+l,q}(X)$. This proves the statement for $p>n$ and the other statement follows by Hodge symmetry.
\end{proof}

\section{Universality of the Kuga-Satake correspondence and applications  \label{secKS}}
We start with an effective rational Hodge   structure $(H^2,F^iH^2_\mathbb{C})$ of weight $2$ with $h^{2,0}=1$ equipped with a symmetric nondegenerate intersection pairing $(\,,\,)$ satisfying the conditions
$$ (\sigma,F^1H^2)=0,\,(\sigma,\overline{\sigma})>0,$$
where $\sigma $ generates $H^{2,0}$.
Note that $(\,,\,)$ satisfies only part of the Hodge-Riemann relations so that the Hodge structure is not in general polarized. We will call such data a Hodge structure of hyper-K\"{a}hler type (although it also encodes the quadratic form) because  this is the structure that we have on the degree $2$ cohomology
$H^2(X,\mathbb{Q})$ of a hyper-K\"{a}hler manifold equipped with   the Beauville-Bogomolov form $q_X$.
The Kuga-Satake correspondence first constructed in  \cite{Kugasatake} associates to a Hodge structure $H^2$ of hyper-K\"{a}hler type a
weight $1$ Hodge structure $H^1_{KS}$, which has the property
that there is an injective morphism
of Hodge structures
$$ H^2\rightarrow {\rm End}  \,H^1_{KS},$$
of bidegree $(-1,-1)$. Note that the Hodge structure on both sides  has Hodge level $2$.
When the Hodge structure  hyper-K\"{a}hler type on $H^2$ is polarized by $(\,,\,)$, which means that our data have  has the extra property that
the pairing $(\,,\,)$ restricted to $H^{1,1}_\mathbb{R}$  is negative definite, the Hodge structure on $H^1_{KS}$ is polarized, hence is the Hodge structure on the $H^1$ of an abelian variety.

The construction of $H^1_{KS}$ can be summarized as follows:
The $\mathbb{Q}$-vector space $H^1_{KS}$ is defined as ${\rm Cliff}\,(H^2,(\,,\,))$, that is, it is the quotient of the tensor algebra
$\otimes H^2$ by the ideal generated by the relations $x^2=(x,x)1$, $x\in H^2$.
The weight $1$ Hodge structure on $ H^1_{KS}$ is given by a complex structure on the real vector space $ H^1_{KS,\mathbb{R}}$. It is constructed as follows. Consider the subspace $(H^{2,0}\oplus H^{0,2})_{\mathbb{R}}\subset H^2_\mathbb{R}$. It is of dimension $2$, naturally oriented, and the restriction of the  form
$(\,,\,)$ to this real plane  is positive definite. Choose a positively oriented orthonormal basis $(e_1,\,e_2)$ of $(H^{2,0}\oplus H^{0,2})_{\mathbb{R}}$. Then $e:=e_1e_2\in {\rm Cliff}\,(H^2_\mathbb{R},(\,,\,))$ does not depend on the choice of basis and satisfies $e^2=-1$. Left Clifford multiplication by $e$ thus defines the desired complex structure on $ {\rm Cliff}\,(H^2_\mathbb{R},(\,,\,))= H^1_{KS,\mathbb{R}}$.

Clifford multiplication on the left induces a morphism
$$ H^2\rightarrow {\rm End}\,H^1_{KS}$$
which is a morphism of Hodge structures of bidegree $(-1,-1)$. This is equivalent to saying that Clifford multiplication on the left by
$H^{1,1}$ preserves the Hodge decompostion of $H^{1}_{KS,\mathbb{C}}$ and that Clifford multiplication on the left by
$H^{1,1}$ shifts the Hodge decomposition of $H^{1}_{KS,\mathbb{C}}$  by $(-1,-1)$. The first fact follows because $H^{1,1}$ is orthogonal to $H^{2,0}$ for $(\,,\,)$, so multiplication by elements of $H^{1,1}$ anticommutes with Clifford  multiplication by elements of $H^{2,0}$ or $H^{0,2}$, hence commutes with Clifford multiplication by $e_1e_2$.  The second fact is an easy computation.

The weight $1$ Hodge structure $H^1_{KS}$ is not simple. In fact it has a big algebra of endomorphisms  given by right Clifford multiplication on the Clifford algebra. These endomorphisms  obviously commute with left Clifford multiplication by $e$, hence provide automorphisms of Hodge structure of $H^1_{KS}$.
To start with, we can restrict the construction to the even Clifford algebra $C^+(H^2, (\,,\,))$ generated by the tensor products
$v_1\otimes\ldots\otimes v_k$, $v_i\in H^2$, with $k$ even,  which  clearly provides  a Hodge substructure of $H^{1}_{KS,\mathbb{Q}}$ since multiplication on the left by $e$ preserves $C^+(H^2_\mathbb{R}, (\,,\,))$. We can do similarly with the odd part $C^-(H^2, (\,,\,))$ of the Clifford algebra, which provides another Hodge substructure.   Multiplication on the right by a given  element $v_0\in H^2$ with $(v,v)\not=0$ provides an isomorphism
$$ C^+(H^2, (\,,\,))\cong C^-(H^2, (\,,\,)),$$
so that, denoting $H^1_{KS+}$, $H^1_{KS-}$ the weight $1$ Hodge structures so obtained,
we have an isomorphism
$$H^1_{KS+}\rightarrow H^1_{KS-}$$
given by right Clifford multiplication by $v_0$, and we get
 an injective  (but not canonical) morphism of Hodge structures
$$ H^2\rightarrow {\rm End}\,H^1_{KS+}$$
given by
$$v\mapsto (\alpha\mapsto v\alpha v_0).$$

Using representation theory of the orthogonal group,  one can  describe up to isogeny and  when the Mumford-Tate group of the Hodge structure on $H^2$ is the  orthogonal group $O((\,,\,))$, the complex tori appearing as subquotient of the  Kuga-Satake complex torus (see \cite{charles}, \cite{geemen}). Note that in the geometric case, it follows from the local  surjectivity of the period map that the Mumford-Tate group is the  orthogonal group $O((\,,\,))$. When $h$ is odd, the Kuga-Satake complex torus  is a power of a simple torus of    dimension $2^{ \frac{h-3}{2}}$ or $2^{\frac{h-1}{2}}$. When    $h$ is even, the situation is much more delicate, as the classification of the subquotients depends on the discriminant of the quadratic form $(\,,\,)$. In this case, the Kuga-Satake complex torus is a sum of powers of one or two  simple complex tori which can be of dimension $2^{h/2},\,2^{h/2-1}$ or $2^{h/2-2}$.
The numbers above  are obtained starting from the fact that
the even Clifford algebra $C^+(H^2)$ has dimension $2^{h-1}$ and that the action on it  by right multiplication by elements of $C^+(H^2)$ (which are morphisms of Hodge structures since they commute with the left multiplication by $e$)  splits  it as a direct sum of weight $1$  Hodge structures.
We now   consider  the polarized case.  
The  Kuga-Satake Hodge structure  $H^1_{KS+}$ is then  polarized and thus is the weight $1$ Hodge structure on the degree $1$ rational cohomology of an abelian variety, that we will denote $KS(H^2,(\,,\,))$, and is defined up to isogeny. The two dual complex tori appearing above are then isomorphic. In the case where $h$ is even, the simple abelian variety one gets has a quadratic endomorphism which makes it a Weil abelian variety.
In the geometric case, where we start from the degree $2$ cohomology of a hyper-K\"{a}hler manifold $X$, equipped with the Beauville-Bogomolov form $ (\,,\,)_{BB}$, we assume $X$ is polarized by an ample class $l\in{\rm NS}(X)$ and put $$(H^2,(\,,\,))=(H^2(X,\mathbb{Q})^{\perp l},q_X)$$
$$KS(X)=KS(H^2,(\,,\,)).$$

The following universality property is proved in \cite{charles}  (see also  \cite{vgeemen-voisin} for a  slightly different   statement, proved only in the polarized case).
\begin{theo} \label{theouniv} Let $(H^2,(\,,\,))$ be a  Hodge structure of hyper-K\"{a}hler type. Assume the Mumford-Tate group of $H^2$ is $SO(H^2,(\,,\,))$.  Let $H$ be a simple  effective weight $1$ Hodge structure, such that for nonnegative integers $a,\,b$ of the same parity, there exists an injective morphism of   Hodge structures of bidegree $(\frac{a-b}{2}-1, \frac{a-b}{2}-1)$
$$  H^2\hookrightarrow H^{\otimes a}\otimes ({H^\vee})^{\otimes b}.$$
Then $H$ is a subquotient  of the Kuga-Satake Hodge structure $H^1_{KS+}$. In particular
$${\rm dim}\,H\geq 2^k,\,\,{\rm  where }\,\,k= \frac{h-1}{2}\,\,{\rm if} \,\,h\,\,{\rm  is\,\, odd},  \,\,k= \frac{h-2}{2}\,\,{\rm if}\,\, h\,\,{\rm  is \,\,even}.$$
\end{theo}
If $h$ is  divisible by $4$, the last  inequality can be improved to ${\rm dim}\,H\geq 2^{\frac{h}{2}}$. 

 A first application of this universality property  (or rather a variant of it)  was given in \cite{vgeemen-voisin} where we proved the Matsushita conjecture on the moduli map for Lagrangian fibration of projective hyper-K\"{a}hler manifolds, at least in the case where $b_2(X)\geq 5$, assuming the Mumford-Tate group is maximal. A second application (also in the projective case, with $a=b=1$) was given  by O'Grady in \cite{ogrady}. Let $X$ be a projective hyper-K\"{a}hler  manifold of generalized Kummer deformation type and dimension  $\geq 4$. One has $b_2(X)=7$, hence for a very general projective such hyper-K\"{a}hler manifold, $b_2(X)_{tr}=6$, so that $KS(X)$ is isogenous to a sum of copies of a simple  abelian fourfold  of Weil type. Using Theorem \ref{theosurj3} (that he had  proved by  an explicit computation  in that case), the fact that $b_3(X)=8$, and the universality property of Theorem \ref{theouniv},
O'Grady proved the following result.
\begin{theo}  \label{theoogrady} The intermediate Jacobian $J^3(X)$ of a projective hyper-K\"{a}hler  manifold of generalized Kummer deformation type with $\rho(X)=1$  is a Weil abelian fourfold. The Kuga-Satake variety of $(H^2(X,\mathbb{Q})_{tr},q_X)$ is isogenous to a  sum of two  copies of $J^3(X)$.
\end{theo}

\subsection{Applications to Betti numbers}
In this section, we are going to apply the previous results to get inequalities involving the  Betti numbers of hyper-K\"{a}hler manifolds.
\begin{theo}\label{theob3} Let $X$ be a hyper-K\"{a}hler manifold. Assume that $b_3(X)\not=0$. Then
\begin{eqnarray}\label{eqineq3} b_3(X)\geq 2^k,
\end{eqnarray}
where    $k= \frac{b_2(X)-1}{2}$ if $b_2(X)$ is odd, $k= \frac{b_2(X)-2}{2}$ if $b_2(X)$ is even.

If $b_2(X)$ is divisible by $4$, the last inequality can be improved to $b_3(X)\geq 2^{b_2(X)/2}$.
\end{theo}

\begin{proof} By Theorem \ref{theosurj3}, we have a surjective morphism of Hodge structures
$$\bigwedge^2H^3(X,\mathbb{Q})\rightarrow H^{4n-2}(X,\mathbb{Q})\cong H^2(X,\mathbb{Q})^*,$$
which gives as well an injective morphism of Hodge structures

$$ H^2(X,\mathbb{Q})\hookrightarrow  \bigwedge ^2H^3(X,\mathbb{Q})^*\hookrightarrow H^3(X,\mathbb{Q})^*\otimes H^3(X,\mathbb{Q})^*.$$

Choosing the complex  structure on $X$  very general so that the Mumford-Tate group of the Hodge structure on $H^2(X,\mathbb{Q})$ is the orthogonal group of $(\,,\,)$, we can thus apply Theorem \ref{theouniv}, which gives (\ref{eqineq3}).
\end{proof}
We now turn to the Betti number $b_{2n-1}$. We prove the following
\begin{theo}\label{theob2n-1} Let $X$ be a hyper-K\"{a}hler manifold such that
$ H^{2n-3}(X)=0$ and $ H^{2n-1}(X)\not=0$. Then
\begin{eqnarray}\label{eqineq2n-1} b_{2n-1}(X)\geq 2^k,
\end{eqnarray}
where   $k= \frac{b_2(X)-1}{2}$ if $b_2(X)$ is odd,  $k= \frac{b_2(X)-2}{2}$ if $b_2(X)$ is even.

If $b_2(X)$ is divisible by $4$, the last inequality can be improved to $b_{2n-1}(X)\geq 2^{b_2(X)/2}$.
\end{theo}
\begin{proof}  By Corollary \ref{corsur2n-1}, the cup-product map
$$\bigwedge^2 H^{2n-1}(X,\mathbb{Q})\rightarrow H^{4n-2}(X,\mathbb{Q})$$
is surjective.
As we assumed $H^{2n-3}(X,\mathbb{Q})=0$, the Hodge structure on  $H^{2n-1}(X,\mathbb{Q})$ has Hodge level $1$ by Lemma
\ref{ledecomp}.
We are thus exactly as in the  situation of Theorem \ref{theob3}  and the same  arguments give inequality (\ref{eqineq2n-1}).
\end{proof}
In the case of a hyper-K\"{a}hler manifold $X$ of dimension $2n=6$, we get
\begin{coro} Let $X$ be a hyper-K\"{a}hler $6$-fold such that
$ H^{{\rm odd}}(X)\not=0$. Then
\begin{eqnarray}\label{eqineqodd} b_{odd}(X)\geq 2^k,
\end{eqnarray}
where  $k= \frac{b_2-2}{2}$ if $b_2$ is even,  $k= \frac{b_2-1}{2}$ if $b_2$ is odd.
\end{coro}
\begin{proof} We observe that, in dimension $6$, if $H^{\rm odd}(X)\not=0$, then either $H^3(X,\mathbb{Q})\not=0$ or, $H^3(X,\mathbb{Q})=0$ and $H^5(X,\mathbb{Q})\not=0$. In the first case we apply  Theorem \ref{theob3} and in the second case we apply Theorem  \ref{theob2n-1}.
\end{proof}
\section{Algebraicity of the Kuga-Satake correspondence}
Let $X$ be a projective complex manifold. Assume that $h^{2,0}(X)=1$, so that the Hodge structure on $H^2(X,\mathbb{Q})$ is of hyper-K\"{a}hler type (choosing a polarization $l$ on $X$, the Lefschetz intersection pairing $(\,,\,)_{\rm lef}$ defined by
$$ (\alpha,\beta)_{\rm lef}=\int_Xl^{n-2}\alpha\cup \beta$$
gives  the desired intersection form). In the case of a hyper-K\"{a}hler manifold of dimension $2n$, the Beauville-Bogomolov intersection pairing on $H^2(X,\mathbb{Q})$ is independent of the choice of a polarization, but when we restrict it to the $l$-primitive cohomology
$H^2(X,\mathbb{Q})_{\rm prim}=H^2(X,\mathbb{Q})^{\perp l^{2n-1}}$, the two pairings coincide up to a scalar coefficient.
Let  $KS(X)$ be the Kuga-Satake abelian variety (defined up to isogeny) associated to the polarized Hodge structure  $(H^2(X,\mathbb{Q})_{\rm prim}, (\,,\,)_{\rm lef})$.
Using an adequate  polarization on  $KS(X)$,  the injective morphism of Hodge structures
$$  H^2(X,\mathbb{Q})_{\rm prim}\hookrightarrow {\rm End}\, (H^1_{KS+}(H^2(X,\mathbb{Q})_{\rm prim},(\,,\,)))= {\rm End} \,H^1(KS(X),\mathbb{Q})$$
 gives an injective  morphism of Hodge structures
\begin{eqnarray} \label{eqmorohKS}  H^2(X,\mathbb{Q})_{\rm prim}\hookrightarrow  H^1(KS(X),\mathbb{Q})^{\otimes 2} .\end{eqnarray}
whose image is contained in $\bigwedge^2 H^1(KS(X),\mathbb{Q})=H^2(KS(X),\mathbb{Q})$.

A morphism of Hodge structures  $\beta: H^2(X,\mathbb{Q})_{\rm prim}\rightarrow H^2(KS(X),\mathbb{Q})$ as in (\ref{eqmorohKS}) provides a Hodge class  (see \cite{voisintorino})

\begin{eqnarray}\label{eqhodgeclas} \alpha\in H^{4n-2}(X,\mathbb{Q})\otimes  H^2(KS(X),\mathbb{Q})\subset H^{4n}(X\times KS(X) ,\mathbb{Q}).
\end{eqnarray}

The Hodge conjecture thus predicts  that there is a cycle $\Gamma\in {\rm CH}^{2n}(X\times KS(X))_\mathbb{Q}$ such that $[\Gamma]=\alpha$, hence in particular
$$[\Gamma]_*=\beta:  H^2(X,\mathbb{Q})_{\rm prim}\rightarrow H^2(KS(X),\mathbb{Q}).$$
When this holds, we will say that the Kuga-Satake correspondence is algebraic.

In the case where $X$ is  an abelian surface, so $b_{2}(X)_{tr}\leq 5$, or more generally any projective  $K3$ surface with $\rho\geq 17$, the algebraicity of the class $\alpha$ above is proved by Morrison \cite{morrison}. In that case, the Kuga-Satake variety is isogenous to a sum of copies of  the abelian surface itself.

In the next case, where $b_{2,tr}=6$, we already mentioned that the Kuga-Satake variety is isogenous to a sum of copies of a $4$-dimensional abelian variety which is of  Weil type (assuming the maximality of the Mumford-Tate group).  This case appears geometrically with  $K3$ surfaces with Picard number $16$ and the first  family of such $K3$ surfaces for which the Kuga-Satake correspondence was known to be algebraic was found by Paranjape \cite{paranjape}. The Paranjape $K3$ surfaces  are obtained by desingularizing  double covers of $\mathbb{P}^2$ ramified along the union of six lines.

The geometric situation we consider is the same as in \cite{ogrady}, \cite{markman}. $X$ is a projective hyper-K\"{a}hler manifold of generalized Kummer type. In particular, we know by O'Grady theorem (Theorem \ref{theoogrady}) that $J^3(X)$ is isogenous to a component of the Kuga-Satake variety $KS(X)$.
We prove now the following result.
\begin{theo}\label{theoKS}  Let $X$ be a projective hyper-K\"{a}hler manifold of generalized Kummer type. Then the Kuga-Satake correspondence of $X$ is algebraic.
\end{theo}
This theorem should be actually considered as an addendum  to Markman's paper \cite{markman}. The result  will indeed follow  from the following result (Theorem \ref{theomarkman})  of Markman.
As we already mentioned, for $X$ as above, the Hodge structure on $H^3(X,\mathbb{Q})$ is of Hodge level $1$,  that is,   of type $(2,1)+(1,2)$.
The generalized Hodge conjecture thus predicts that the degree $3$ cohomology of $X$ is supported on a (singular) divisor of $X$, and this is equivalent
to the fact that the Griffiths Abel-Jacobi map
\begin{eqnarray}\label{eqphiX} \Phi_X: {\rm CH}^2(X)_{\rm alg}\rightarrow J^3(X)
\end{eqnarray}
is surjective (see \cite{voisintorino}).
\begin{theo} \label{theomarkman} (Markman \cite{markman})  For a projective hyper-K\"{a}hler manifold of Kummer deformation type, the Abel-Jacobi map (\ref{eqphiX}) is surjective.
\end{theo}
\begin{proof}[Proof of Theorem \ref{theoKS}] An equivalent version of  Theorem \ref{theomarkman} says that  there exists a codimension cycle $\mathcal{Z}\in {\rm CH}^2(J^3(X)\times X)_\mathbb{Q}$ such that
the map $[\mathcal{Z}]_*:H_1(J^3(X),\mathbb{Q})\rightarrow H^3(X,\mathbb{Q})$ is the natural identification
$H_1(J^3(X),\mathbb{Q})\cong   H^3(X,\mathbb{Q})$. We recall here that $J^3(X) $ is the complex torus
$H^3(X,\mathbb{C})/(F^2H^3(X,\mathbb{C})\oplus  H^3(X,\mathbb{Z}))$ built from the Hodge structure on $H^3(X,\mathbb{Z})$ so that
$H_1(J^3(X),\mathbb{Z})=H^3(X,\mathbb{Z})$ canonically.
Note that we can  assume that the cohomology class $[\mathcal{Z}]\in H^4(J^3(X)\times X, \mathbb{Q})$ belongs to the K\"{u}nneth component
$H^1(J^3(X),\mathbb{Q} )\otimes  H^3(X, \mathbb{Q})$. Indeed, using the action of the  maps of multiplication by $k$  on $J^3(X)$, the K\"{u}nneth components of $[\mathcal{Z}]$ are all algebraic, and the K\"{u}nneth components not in  $H^1(J^3(X),\mathbb{Q} )\otimes  H^3(X, \mathbb{Q})$
induce the zero map $H_1(J^3(X),\mathbb{Q})\rightarrow H^3(X,\mathbb{Q})$.

Next, by another result of Markman \cite{markmanBBcahar}, the class $Q_X\in H^4(X,\mathbb{Q})$ introduced in (\ref{eqQX}) is algebraic on hyper-K\"{a}hler manifolds of generalized Kummer type. It is thus the class of a cycle $\mathcal{Q}_X\in{\rm CH}^2(X)_\mathbb{Q}$.
On $J^3(X)\times X$, we consider the following cycle
\begin{eqnarray} \Gamma:=\mathcal{Z}^2\cdot {\rm pr}_X^*\mathcal{Q}_X^{2n-2},
\end{eqnarray}
where ${\rm pr}_X: J^3(X)\times X\rightarrow X$ denotes the second projection.
We prove the following
\begin{claim}\label{claim} The map $[\Gamma]_*: H_2(J^3(X),\mathbb{Q})\rightarrow H^{4n-2}(X,\mathbb{Q})$
identifies with the O'Grady map
$\phi: \bigwedge^3H^3(X,\mathbb{Q})\rightarrow H^6(X,\mathbb{Q})\stackrel{Q_X}{\rightarrow} H^{4n-2}(X,\mathbb{Q})$ of (\ref{eqogradymap}).
\end{claim}
\begin{proof} Recall that we assumed that $[\mathcal{Z}]\in  H^1(J^3(X),\mathbb{Q} )\otimes  H^3(X, \mathbb{Q})$. Taking a basis $e_i$ of
$H^3(X, \mathbb{Q})$, which provides a basis $f_i$ of $H_1(J^3(X),\mathbb{Q})$ and the dual basis
$f_i^*$ of $H_1(J^3(X),\mathbb{Q})$, we can thus write
\begin{eqnarray} \label{eqecriZ} [\mathcal{Z}]=\sum_i {\rm pr}_{J^3(X)}^*f_i^*\cup {\rm pr}_X^* e_i,
\end{eqnarray}
since $[\mathcal{Z}]_*(f_i)=e_i$.
We now deduce from (\ref{eqecriZ})
$$[\Gamma]=\sum_{i,j}{\rm pr}_{J^3(X)}^*f_i^*\cup f_j^*\cup {\rm pr}_X^*e_i\cup   {\rm pr}_X^*e_j\cup   {\rm pr}_X^*Q_X,$$
which immediately implies the claim. \end{proof}
The claim implies the theorem  since we already identified the intermediate Jacobian with a component of the Kuga-Satake variety, in such a way that the transpose of the map (\ref{eqmorohKS}) is the O'Grady map. Thus the map  (\ref{eqmorohKS}) or its transpose  is induced by an algebraic cycle.
\end{proof}

CNRS, Institut de Math\'{e}matiques de Jussieu-Paris rive gauche

claire.voisin@imj-prg.fr
    \end{document}